\documentclass{amsart}

\usepackage{graphicx}  
\usepackage{amsmath}  
\usepackage{amssymb,amsthm}
\usepackage{color}
\usepackage{tikz}
\usepackage{hyperref}

\newcommand{\R}{{\mathbb R}}\newcommand{\N}{{\mathbb N}}

\let\epsilon\varepsilon
\let\theta\vartheta

\newtheorem{theorem}{Theorem}[section]\newtheorem{lemma}[theorem]{Lemma}

\newtheorem{corollary}[theorem]{Corollary}

\newtheorem{remark}[theorem]{Remark}

\newcommand{\dist}{\operatorname{dist}}

\begin{document}
\title[Diffusive stability]{Diffusive stability and self-similar decay \\ for the harmonic map heat flow}

\author[T.~Lamm]{Tobias Lamm}
\address[T.~Lamm]{Institute for Analysis\\
Karlsruhe Institute of Technology \\ Englerstr. 2\\
76131 Karlsruhe\\ Germany}
\email{tobias.lamm@kit.edu}

\author[G.~Schneider]{Guido Schneider}
\address[G.~Schneider]{Institut f\"ur Analysis, Dynamik und Modellierung \\  Universit\"at Stuttgart\\  Pfaffenwaldring 57 \\  70569 Stuttgart, Germany}
\email{guido.schneider@mathematik.uni-stuttgart.de}

\thanks{We thank Peer Kunstmann for many helpful discussions on Besov spaces.
Funded by the Deutsche Forschungsgemeinschaft (DFG, German Research Foundation) - Project-ID 258734477 - SFB 1173}

\date{\today{}}

\begin{abstract}
In this paper we study the harmonic map heat flow on the euclidean space $\R^d$ and we show an unconditional uniqueness result for maps with small initial data in the homogeneous Besov space $\dot{B}^{\frac{d}{p}}_{p,\infty}(\R^d)$ where $d<p<\infty$. As a consequence we obtain decay rates for solutions of the harmonic map flow of the form $\|\nabla u(t) \|_{L^\infty(\R^d)}\leq Ct^{-\frac12}$.

Additionally, under the assumption of a stronger spatial localization of the initial conditions, we 
show that the temporal decay happens in a self-similar way.
We also explain that similar results hold for the biharmonic map heat flow and the semilinear heat equation with a power-type nonlinearity.
\end{abstract}
\maketitle

\section{Introduction}

The diffusive stability method is a tool which
allows us to prove stability results for parabolic partial differential equations in the case of a linearization possessing essential
spectrum up to the imaginary axis. It is based on algebraic decay rates 
which hold for the 
heat semigroup $ (e^{t \Delta})_{t \geq 0} $ on $\R^d$, defined by
$$
(e^{t \Delta} u_0)(x) = \frac{1}{(4 \pi t)^{d/2} } \int_{\R^d}
e^{-\frac{|x-y|^2}{4t}} u_0(y) dy.
$$
We have the decay
$$ \| e^{t \Delta} \|_{L^p \to L^{\infty}} 
\leq C t^{-\frac{d}{2p}} ,$$
for $ p \in [1,\infty) $,
due to the diffusion which occurs for spatially localized initial conditions and 
the decay 
$$ \| e^{t \Delta} \nabla \|_{L^{\infty} \to L^{\infty}}  \leq C t^{-\frac{1}{2}} $$
due to smoothing. These decays can be used  for instance to prove  stability results
in $ (L^p \cap C^1_b) (\R^d)$ 
of  the trivial solution $ u = 0 $ in semilinear heat equations
$$ 
\partial_t u = \Delta u + f(u,\nabla u) ,
$$  
with $ x \in \R^d $, $ t \geq 0 $, $ u(x,t) \in \R $, 
and  a smooth nonlinear function $ f $ containing terms of the form $ u^{m} (\partial_{1} u)^{m_1} \ldots  
(\partial_{d} u)^{m_d} $ and satisfying $f(0,0)=0$.
Since $ \partial_t u \sim  t^{-\frac{d}{2p}-1} $ and $  \Delta u \sim  t^{-\frac{d}{2p}-1}  $
for $ t \to \infty $ 
we have that the nonlinear terms behave like
\[
u^{m} (\partial_{1} u)^{m_1} \ldots  
(\partial_{d} u)^{m_d}  \sim  t^{- \frac{(m+m_1+\ldots + m_d)d}{2p} - \frac{m_1+\ldots + m_d}{2}} 
\]
for $ t \to \infty $. Hence, they 
are asymptotically irrelevant with respect to the
linear diffusion if 
\begin{equation*} 
\frac{d}{2p}+ 1 <  \frac{(m+m_1+\ldots + m_d)d}{2p} + \frac{m_1+\ldots + m_d}{2}.
\end{equation*}
As an example for $ d= 1 $ and  $ p= 1 $ all such terms are irrelevant except for $ u^2 $, $ u^3 $,
and $ u \partial_x u $. With this idea the stability of $ u = 0 $ in $ (L^p \cap C^1_b)(\R^d) $ follows
for 
$$ 
\partial_t u = \Delta u +  u^{m} (\partial_{x_1} u)^{m_1} \ldots  
(\partial_{x_d} u)^{m_d} ,
$$  
if the above condition is satisfied. An abstract existence result in this direction can be found in \cite{Uec99}.

After the introduction of this method in the papers of Fujita and Weisler \cite{Fujita,Wei81},
the last decades saw a successful application of
this tool to stability questions arising for pattern
forming systems, see e.g. \cite{BKL94,Schn98ARMA,SSSU12JDE,JNRZ14}. A summary and overview of various aspects of this diffusive stability  method can also be found in the monograph \cite{SU}. 

It is the goal of this paper 
to use  this tool for gaining new stability and uniqueness results
for dissipative semilinear 
geometric flow problems on $ \R^d $, in particular for the harmonic map heat flow.
\newline
\newline
In the following we let $N$ be a smooth Riemannian manifold which we assume to be isometrically embedded into some euclidean space $\R^m$. For $\delta>0$ sufficiently small we have a well defined nearest point projection $\pi:N_\delta\to N$, where $N_\delta:=\{ x\in \R^m:\, \dist(x,N)<\delta\}$ and we define the second fundamental form $A(y):\, T_y N \times T_yN \rightarrow T^\perp_y N$ for every $y\in N$ and $v_1,v_2\in T_yN$ by $A(y)(v_1,v_2)=-D^2 \pi(y)(v_1,v_2)$.
A map $u:\R^d \times [0,T)\to N$ is then a solution of the harmonic map heat flow with initial data $u(\cdot,0)=u_0:\R^d\to N$ if it solves the parabolic partial differential equation
\begin{align}
\partial_t u-\Delta u =A(u)(\nabla u,\nabla u).\label{start1}
\end{align}
For instance, in case of $ N = S^{m-1}  = \{ x\in \R^m : \| x \|_{\mathbb{R}^{m}} = 1 \}$, we compute 
\begin{eqnarray*}
	A(u)(\nabla u,\nabla u) &= &- 
	(\Delta u,u) u= -(\partial_k^2 u_\alpha) u_\alpha  u
	\\ & = &  -\partial_k ((\partial_k u_\alpha)u_\alpha)u
	+ (\partial_k u_\alpha)( \partial_k u_\alpha)u
	= u |\nabla u|^2,
\end{eqnarray*}
where here and in the following we use Einstein's summation convention and the fact that $ (\partial_k u_\alpha)u_\alpha = 0 $ which follows from the fact that $ u_\alpha u_\alpha = 1 $.

Rather general existence results for the harmonic map heat flow on euclidean space have been obtained by Koch and the first author \cite{KochLamm} for initial maps with a small oscillation and by Wang \cite{wang11} for initial maps with a small $BMO$-norm. These solutions were constructed by applying a fixed point argument on suitably constructed Banach spaces. As a byproduct of this method, the authors obtained a conditional uniqueness result for the harmonic map heat flow under the smallness assumptions on the initial data mentioned above. 

One of our main goals of this paper is to use the diffusive stability method in order to improve this conditional uniqueness result to an unconditional uniqueness result. The price we have to pay is that our initial data have to be small in a strict subspace of the above mentioned space of functions of bounded mean oscillation. More precisely, we require our initial data to be small in a homogeneous Besov space, see Theorem \ref{main1} for further details.

The deviation $ v $ of a trivial spatially constant equilibrium $ u_* $ satisfies a semilinear diffusion equation of the form 
\begin{equation}
\label{veq}
\partial_t v-\Delta v =A(u^*+v)(\nabla v,\nabla v).
\end{equation}
As above we have $ A(u^*+v)(\nabla v,\nabla v) \sim
t^{- \frac{d}{p} - 1}  $  for $ t \to \infty $, and so the nonlinearity 
is irrelevant w.r.t. to linear diffusion if $ p \in [1,\infty) $.
For $ p = \infty $ the linear and nonlinear terms are of the same order and a 
stability result on exponentially long time scales can be established, see Theorem \ref{thsec3b}.

For slightly stronger localized initial conditions the associated solutions 
of the linear heat equation decay in a self-similar way towards zero.
The same is true for semilinear heat equations with nonlinearities which are irrelevant in 
$ L^1 \cap C^1_b (\R^d)$ with the above counting. 
Using the discrete renormalization approach, cf. \cite{BK92,BKL94}, gives 
	that 
	the  renormalized solution
	$  t^{d/2} v(x \sqrt{t} ,t) $ of \eqref{veq} converges towards a 
	multiple of a Gaussian $ e^{-|x|^2/4}$
	for $ t \to \infty $. In the discrete renormalization approach
	instead of solving the PDE \eqref{veq} directly we consider an equivalent  sequence of 
	problems which converges formally towards the linear diffusion equation.
	Let
\begin{equation*} 
v_n(x,\tau) = L^{d n}v(L^n x,L^{2n}\tau),
\end{equation*}
with  $L > 1$ fixed and  $n \in \N$. Then  $v_n$ satisfies 
\begin{eqnarray}
\label{kass}
\partial_{\tau} v_n- \Delta v_n &= & L^{n(d + 2 - 2 - 2d)} D^2\pi( u^*+L^{-d n}v) (\nabla v_n,\nabla v_n) \\ 
&= & L^{-n d} D^2\pi( u^*+L^{-d n}v) (\nabla v_n,\nabla v_n) \nonumber
\end{eqnarray}
for $ \tau \in [L^{-2},1] $,
with
\begin{equation}
v_n(x,L^{-2})=L^d v_{n-1}(L x,1).
\end{equation}
The influence of the nonlinear terms vanishes 
for  $n \to \infty$ with a geometric rate. In the limit we obtain 
a linear diffusion equation.
With a simple  fixed point argument  
the convergence of  
$v_n|_{\tau = 1} $ towards the limit of the renormalized
linear diffusion  problem, namely a multiple of the  Gaussian $ e^{-|x|^2/4}$
can be established, cf. Theorem \ref{th41}.

A different approach for the harmonic map heat flow is the consideration of $ w = \nabla v $
which satisfies 
\begin{equation}
\label{weq}
\partial_t w-\Delta w =\nabla (A(u^*+v)( w,w)). 
\end{equation}
We show that the above mentioned discrete renormalization approach can also be applied to this system of equations.

Finally, we show that our technique can be extended to various other semilinear heat equations on euclidean space, such as equations with a power-type nonlinearity or higher order equations, such as the biharmonic map heat flow.

A similar handling of dissipative quasilinear geometric flow problems,
such as  the
Ricci-DeTurck flow, the mean curvature flow, the Yamabe flow, and the Willmore flow of graphs,
will be the topic of future research.

\medskip

{\bf Notation.}
In the following many possibly different constants are denoted with the same symbol $ C $ if they can be chosen independent of time $ t $.

\section{Stability and uniqueness results for the harmonic map flow}

\label{sec2}
Our first main result about the harmonic map heat flow is given by an asymptotic stability and uniqueness result assuming smallness conditions on a certain homogeneous Besov norm of the inital map $u_0$. 

We note that due to the unboundedness of the domain the linearization 
around $ u= const. $  for \eqref{start1} possesses essential spectrum up to the imaginary axis and so the principle of linearized stability does not apply directly in this situation. Therefore, we use the diffusive stability method for establishing the following results. We start by extending a stability result for the harmonic map flow which was originally due to Soyeur \cite{Soyeur}. Using the notation $\dot{W}^{1,p}(\R^d,N)$ for the homogeneous Sobolev space, his result is as follows
\begin{theorem}\label{soy}
There exists a constant $\varepsilon>0$ depending only on $d$ and $N$ such that if $u_0\in  \dot{W}^{1,p}(\R^d,N)$ with $d<p<\infty$ and $\|\nabla u_0\|_{L^d} <\varepsilon$, then we have the estimate
\begin{align}
t^{\frac12-\frac{d}{2p}}\|\nabla u(t)\|_{L^p}  \leq C\|\nabla u_0\|_{L^d}, \label{pdecay}
\end{align}
where $C=C(d,m,p)$, for every solution $u\in C^0([0,\infty), \dot{W}^{1,p}(\R^d,N))$ of \eqref{start1}.
\end{theorem}
\begin{remark}
In the same paper Soyeur also showed the existence of a solution $u\in C^0([0,\infty), \dot{W}^{1,p}(\R^d,N))$ for small initial data in the above sense. Additionally, the smallness assumption on the $L^d$-norm of $\nabla u_0$ implies that the $BMO$-norm of $u_0$ is small and hence, by a paper of C. Wang \cite{wang11}, this implies the existence of a global smooth solution of the harmonic map flow. 
\end{remark}
\begin{remark}
This result can be extended to maps from complete Riemannian manifolds for which one has good heat kernel estimates as it was done in Li and Wang \cite{LiWang} for \eqref{pdecay}. Examples of such manifolds include manifolds with positive Ricci curvature.
\end{remark}
The original result of Soyeur already extended a previous result of Struwe \cite{struwe}, who was the first one to show a convergence result for the harmonic map heat flow on $\R^d$ to a constant map with additional decay properties for the first derivative. More precisely, he showed in Theorem 7.1 of the paper mentioned above, that for smooth initial maps $u_0$ so that $\|\nabla u_0\|_{L^\infty}$ is bounded and $\|\nabla u_0\|_{L^2}$ is sufficiently small, the unique smooth solution of the harmonic map heat flow satisfies for $t$ large enough $\|\nabla u(\cdot,t)\|_{L^\infty}\leq Ct^{-\frac12}$ and hence converges to a constant map. 

As a first result we improve Theorem \ref{soy} for initial data satisfying a slightly less restrictive smallness condition. More precisely, for a map $v:\R^d\to \R^m$ we denote by $G(t)v:=G_t \star v$ its heat extension, i.e. $G(t)v$ is a solution of the heat equation with initial data $v$. Here $G(x,t)=\frac{1}{(4\pi t)^{d/2}} e^{\frac{-|x|^2}{4t}}$, $(x,t)\in \R^d\times (0,\infty)$, is the standard heat kernel. For $d<p\leq \infty$ we define the homogeneous Besov space $\dot{B}^{\frac{d}{p}}_{p,\infty}(\R^d)$ as the set of functions for which the norm
\[
\|v\|_{\dot{B}^{\frac{d}{p}}_{p,\infty}}:=\sup_{t>0} t^{\frac12 -\frac{d}{2p}} \|\nabla (G(t)v)\|_{L^p}
\]
is finite. It follows that
\[
\dot{W}^{1,d}(\R^d) \subset \dot{B}^{\frac{d}{p}}_{p,\infty}(\R^d) \subset BMO(\R^d) \subset \dot{B}^0_{\infty,\infty}(\R^d)
\]
for every $d<p< \infty$, where each of these inclusions is strict (see e.g. \cite{Cannone}).

Now we are in a position to state our first main result.
\begin{theorem}\label{newsoy}
There exists a constant $\varepsilon>0$ depending only on $d$ and $N$ such that if $u_0\in \dot{B}^{\frac{d}{p}}_{p,\infty}(\R^d,N)$ with $d<p<\infty$, $2\leq p$ and $\|u_0\|_{\dot{B}^{\frac{d}{p}}_{p,\infty}} <\varepsilon$, then we have the estimate
\[
t^{\frac12-\frac{d}{2p}}\|\nabla u(t)\|_{L^p} \leq C\| u_0\|_{\dot{B}^{\frac{d}{p}}_{p,\infty}},
\]
where $C=C(d,m,p)$, for every solution $u:(0,\infty)\times \R^d\to N$ of \eqref{start1} with 
\[
t^{\frac12-\frac{d}{2p}}\nabla u(t,x) \in C^0([0,\infty), L^{p}(\R^d))
\]
and 
\[
\lim_{t\searrow 0} t^{\frac12-\frac{d}{2p}} \|\nabla u(t)\|_{L^p} =0.
\]
\end{theorem}
\begin{remark}
Note that the existence of such solutions follows for example from the work of Soyeur under the assumption that $u_0\in \dot{W}^{1,p}(\R^d,N)$ with a sufficiently small norm.
\end{remark}
In the proof of this result we need a standard estimate on convolutions of the heat kernel. Note that it relies on Young's convolution inequality.
\begin{lemma}\label{kernel}
Let $G(x,t)=\frac{1}{(4\pi t)^{d/2}} e^{\frac{-|x|^2}{4t}}$, $(x,t)\in \R^d\times (0,\infty)$, be the heat kernel. Using the notation $G(t)f=G(t) \star f$ we have the estimate
\begin{align}
\| \nabla^\sigma G (t)  \|_{\mathcal{L}(L^r,L^q)} \leq C t^{-\frac{d}{2}(\frac{1}{r}-\frac{1}{q})-\frac{\sigma}{2}} , \label{heat1}
\end{align}
where $1\leq r\leq q\leq \infty$, $\sigma\in \N_0$ and $C<\infty$ is a constant. Additionally, we have that for some constant $C<\infty$
\begin{align}
\int_0^t (t-s)^{-\alpha}s^{-\beta} \, ds=C t^{1-\alpha-\beta} \label{heat2}
\end{align}
if both $0<\alpha,\beta<1$.
\end{lemma}
Now we are in a position to prove Theorem \ref{newsoy}.
\begin{proof}[Proof of Theorem \ref{newsoy}:]
We start by noting that
\[
\nabla (u(x,t)-G(t) \star u_0 (x)) =\int_0^t \int_{\R^d} \nabla G(x-y,t-s) (A(u)(\nabla u,\nabla u)(y,s))\, dy\, ds.
\]
Using the estimate \eqref{heat1} with $\sigma=1$, $r=\frac{p}{2}$ and $q=p$ we get for every $ t>0$
\begin{align*}
\|\nabla (u(t)-G(t)u_0)\|_{L^p} \leq& C\int_0^t \| |\nabla u(s)|^2 \|_{L^{\frac{p}{2}}} (t-s)^{-(\frac12 +\frac{d}{2p})} \, ds\\
\leq& C\int_0^t \| \nabla u(s) \|^2_{L^p} (t-s)^{-(\frac12 +\frac{d}{2p})} \, ds.
\end{align*}
Setting 
\[
m(T) =\sup_{0\leq t \leq T} t^{\frac12 -\frac{d}{2p}} \|\nabla u(t)\|_{L^p},
\]
the last inequality implies for all $t>0$
\begin{align*}
 t^{\frac12 -\frac{d}{2p}} \|\nabla u(t)\|_{L^p}\leq& C\|u_0\|_{\dot{B}^{\frac{d}{p}}_{p,\infty}} +Ct^{\frac12 -\frac{d}{2p}} m(T)^2 \int_0^t (t-s)^{-(\frac12 +\frac{d}{2p})} s^{-1+\frac{d}{p}} \, ds\\
 \leq& C\|u_0\|_{\dot{B}^{\frac{d}{p}}_{p,\infty}} + Cm(T)^2,
 \end{align*}
 where we used the identity \eqref{heat2} with $0<\alpha=\frac12+\frac{d}{2p}<1$ and $0<\beta=1-\frac{d}{p}<1$. Note that it is exactly here that we have to use the assumption $d<p<\infty$. 
 
 Taking the supremum over all $0\leq t \leq T$ we thus obtain for every $0<T<\infty$
 \[
 Cm(T)^2 -m(T)+C \|u_0\|_{\dot{B}^{\frac{d}{p}}_{p,\infty}} \geq 0
 \]
 and for $\varepsilon>0$ small enough this finally implies for all $T<\infty$
 \[
 m(T) \leq C\|u_0\|_{\dot{B}^{\frac{d}{p}}_{p,\infty}} \leq C\varepsilon.
 \]
 Here we used that $m(0)=0$. By definition this estimate implies that
 \[
 t^{\frac12- \frac{d}{2p}}\|\nabla u(t)\|_{L^p}  \leq C \|u_0\|_{\dot{B}^{\frac{d}{p}}_{p,\infty}} \]
 for every $0<t<\infty$, as claimed.
\end{proof}

In the next theorem we improve  a uniqueness result for solutions of the harmonic map heat flow with initial data with small $BMO$-norm. 
For this we define the Banach space $X$ as the space of functions such that the norm
\begin{align*}
\|u\|_{X}:=& \sup_{0<t<\infty}\left( \|u(t)\|_{L^\infty(\R^d)} +t^{\frac12} \|\nabla u(t)\|_{L^\infty(\R^d)}\right) \\
&+\sup_{x\in \R^d} \sup_{0<R<\infty} R^{-\frac{d}{2}}\|\nabla u\|_{L^2(B_R(x)\times(0,R^2))}
\end{align*}
is finite.
It was shown by the first author and Koch \cite{KochLamm} that there exist constants $\varepsilon>0$, $C>0$ such that for every $u_0:\R^d \to N$ satisfying $\|u_0-P\|_{L^\infty} <\varepsilon$, where $P\in N$ is some arbitrary point, there exists a global solution $u\in P+X$ of \eqref{start1}. Moreover, the solution is unique in the ball
\[
B^{X}(P,C\varepsilon)=\{ u:\, \|u-P\|_X\leq C\varepsilon\}.
\]
This result was later on extended to initial data $u_0\in BMO(\R^d)$ with small BMO-norm by Wang \cite{wang11}. He obtained the uniqueness of the solution in the set $B^{X}(G(t)u_0,C\varepsilon)$.

As our second main result for the harmonic map flow we improve this conditional uniqueness result to an unconditional uniqueness result (i.e. we no longer have to restrict to the class of solutions in $X$ with small $X$-norm) but the price we have to pay is to go from $BMO$-initial data to homogeneous Besov initial data.
\begin{theorem}\label{main1}
Let $N$ be a smooth and closed manifold. There exists $\varepsilon>0$, $C>0$ such that for every $u_0\in \dot{B}^{\frac{d}{p}}_{p,\infty}(\R^d,N)$ with $d<p<\infty$, $\|u_0\|_{\dot{B}^{\frac{d}{p}}_{p,\infty}} <\varepsilon$ and every solution $u:(0,\infty)\times \R^d\to N$ of \eqref{start1} with 
\[
t^{\frac12-\frac{d}{2p}}\nabla u(t,x) \in C^0([0,\infty), L^{p}(\R^d))
\]
and 
\[
\lim_{t\searrow 0} t^{\frac12-\frac{d}{2p}} \|\nabla u(t)\|_{L^p} =0,
\]
we have the estimate
\[
\| u -G(t)u_0\|_X\leq C\varepsilon,
\]
i.e. there exists only one solution under these assumptions.
\end{theorem}
\begin{remark}
Since 
\[
\sup_{t>0} t^{\frac12} \|\nabla (G(t)u_0)\|_{L^\infty(\R^d)} \simeq \|u_0\|_{\dot{B}^0_{\infty,\infty}},
\]
respectively
\[
\sup_{x\in \R^d} \sup_{0<R<\infty} R^{-\frac{d}{2}}\|\nabla (G(t)u_0)\|_{L^2(B_R(x)\times(0,R^2))} \simeq \|u_0\|_{BMO}
\]
and since we have the inclusions
\[
\dot{B}^{\frac{d}{p}}_{p,\infty} \subset BMO \subset \dot{B}^0_{\infty,\infty}
\]
for every $d<p<\infty$, it follows from the above estimates that for every solution $u\in C^0([0,\infty), \dot{W}^{1,p}(\R^d,N))$ of \eqref{start1} as above with initial data $u_0$ satisfying $\|u_0\|_{\dot{B}^{\frac{d}{p}}_{p,\infty}} <\varepsilon$ we have
\[
\sup_{t>0} t^{\frac12} \|\nabla u \|_{L^\infty(\R^d)} +\sup_{x\in \R^d} \sup_{0<R<\infty} R^{-\frac{d}{2}}\|\nabla u\|_{L^2(B_R(x)\times(0,R^2))} \leq C\|u_0\|_{\dot{B}^{\frac{d}{p}}_{p,\infty}}.
\]
\end{remark}
\begin{proof}[Proof of Theorem \ref{main1}:]
We fix $p>d$ and we note that we get from Theorem \ref{newsoy} the estimate 
\[
 \sup_{t>0} t^{\frac12- \frac{d}{2p}}\|\nabla u(t)\|_{L^p}  \leq C \|u_0\|_{\dot{B}^{\frac{d}{p}}_{p,\infty}}.
\]
Next we use again the estimate \eqref{heat1} with $\sigma=1$, $r=\frac{p}{2}$ and $q=\infty$ to conclude for every $ t>0$
\begin{align*}
\|\nabla (u(t)-G(t)u_0)\|_{L^\infty} \leq& C\int_0^t \| |\nabla u(s)|^2 \|_{L^{\frac{p}{2}}} (t-s)^{-(\frac12 +\frac{d}{p})} \, ds\\
\leq& C \int_0^t \| \nabla u(s) \|^2_{L^p} (t-s)^{-(\frac12 +\frac{d}{p})} \, ds\\
\leq& C  \|u_0\|^2_{\dot{B}^{\frac{d}{p}}_{p,\infty}} \int_0^t s^{\frac{d}{p}-1} (t-s)^{-(\frac12 +\frac{d}{p})} \, ds\\
\leq& C t^{-\frac12}  \|u_0\|^2_{\dot{B}^{\frac{d}{p}}_{p,\infty}},
\end{align*}
where we used \eqref{heat2} in the last estimate.

Similarly, we obtain
 \begin{align*}
\|u(t)-G(t)u_0\|_{L^\infty} \leq& C\int_0^t \| |\nabla u(s)|^2 \|_{L^{\frac{p}{2}}} (t-s)^{-\frac{d}{p}} \, ds\\
\leq& C \int_0^t \| \nabla u(s) \|^2_{L^p} (t-s)^{-\frac{d}{p}} \, ds\\
\leq& C  \|u_0\|^2_{\dot{B}^{\frac{d}{p}}_{p,\infty}} \int_0^t s^{\frac{d}{p}-1} (t-s)^{-\frac{d}{p}} \, ds\\
\leq& C \|u_0\|^2_{\dot{B}^{\frac{d}{p}}_{p,\infty}}.
\end{align*}
It remains to establish the bound 
\[
\sup_{x\in \R^d} \sup_{0<R<\infty} R^{-\frac{d}{2}}\|\nabla (u-G(\cdot)u_0)\|_{L^2(B_R(x)\times(0,R^2))} \leq C \|u_0\|_{\dot{B}^{\frac{d}{p}}_{p,\infty}}.
\]
Using H\"older's inequality, we estimate for every $x\in \R^d$ and some $p>d$
\begin{align*}
\int_{B_R(x)} |\nabla (u(t)-G(t)u_0)|^2 \leq& C(\int_{B_R(x)} |\nabla (u(t)-G(t)u_0)|^p)^{\frac{2}{p}} R^{d(1-\frac{2}{p})}\\
\leq& Ct^{\frac{d}{p}-1} R^{d(1-\frac{2}{p})}  \|u_0\|^2_{\dot{B}^{\frac{d}{p}}_{p,\infty}}.
\end{align*}
From this it is easy to see that
\[
\sup_{x\in \R^d} \sup_{0<R<\infty} R^{-\frac{d}{2}}\|\nabla (u-G(\cdot)u_0)\|_{L^2(B_R(x)\times(0,R^2))} \leq C \|u_0\|_{\dot{B}^{\frac{d}{p}}_{p,\infty}}.
\]
\end{proof}
Finally, we show a new stability result on exponentially long time scales for the harmonic map flow starting near a constant map. Note that in this result we do not assume that the initial values decay in space.
\begin{theorem} \label{thsec3b}
	For  every $C > 0$ there exist $b>0$ and $\varepsilon_0 > 0$ such that for every $\varepsilon\in (0,\varepsilon_0)$ and every
	solution $u:\R^d\times [0,\infty)\to N$ of \eqref{start1} with  $\|u_0-P\|_{W^{1,\infty}} \leq \varepsilon$, where $P\in N$ is an arbitrary point, we have
	 $$ \sup_{t \in [0,\exp\left(\frac{b}{\varepsilon}\right)-1]} \|u(t)-P\|_{C^0}+\sup_{t \in [0,\exp\left(\frac{b}{\varepsilon}\right)-1]} (1{+}t)^{\frac12} \|\nabla u(t)\|_{C^0} \le C .	$$
\end{theorem}
\begin{proof}
 We define the functions
\begin{eqnarray*}
		a(t) & = & \sup_{0\leq s\leq t} \|u(s)-P\|_{L^\infty} \ \ \ \text{resp.} \\
		b(t) & = & \sup_{0\leq s\leq t} \|(1{+}s)^{\frac12} \nabla  u(s)\|_{L^\infty},\\
		\end{eqnarray*} 
and we estimate (using the fact that $\|G(t)(u_0-P)\|_{L^\infty}\leq \|u_0-P\|_{L^\infty}$)
\begin{eqnarray*}
	\|u(t)-P\|_{L^\infty} & \leq& C\|u_0-P\|_{L^\infty}+C\biggl\|\int^t_0 G(\cdot,t-s) \star A(u)(\nabla u,\nabla u) (s)d s\biggr\|_{L^{\infty}}\\
&	\leq& C a(0)+C  \int_0^t \|   |\nabla u(s)|^2\|_{L^\infty} 
	d s\\
&	\leq&  Ca(0)+C b^2(t) \int^t_0 (1{+}s)^{-1} \,d s\\
&	\leq&  Ca(0)+ C b^2(t)  \log(1+t)
	\end{eqnarray*}
	and
	\begin{align*}
	(1+t)^{\frac12} \|\nabla u(t)\|_{L^\infty} \leq& C(a(0)+b(0))+Cb^2(t)(1+t)^{\frac12} \int_0^t (t-s)^{-\frac12} (1+s)^{-1} \, ds\\
	\leq&  C(a(0)+b(0)) + C b^2(t) (1+\log(1+t)),
	\end{align*}
	where we used that
		\[
\|\nabla (G(\cdot,t) \star (u_0-P))\|_{L^\infty} \leq  \begin{cases} C\|\nabla u_0\|_{L^\infty} , & \text{for } \, t\leq 1, \\ Ct^{-\frac{1}{2}} \| u_0-P\|_{L^\infty}, & \text{for } \, t>1. \end{cases}
\] 
Next we define the rescaled functions $\tilde{a}(t)=\epsilon^{-1} a(t)$ resp. $\tilde{b}(t) =\epsilon^{-1} b(t)$ and we conclude
	\[
	\tilde{a}(t)+\tilde{b}(t) \leq C_1 (\tilde{a}(0) +\tilde{b}(0)) + C_2 \epsilon (1+\log(1+t)) \tilde{b}^2(t)
	\]
	for some constants $C_1,C_2>0$. Since $\tilde{a}(0) +\tilde{b}(0)\leq 1$ we conclude
	\[
	\tilde{a}(t)+\tilde{b}(t) \leq 2C_1 
	\]
	as long as
	\[
	1+t\leq \exp{\left(\frac{1}{4C_1C_2 \epsilon}-1\right)}.
	\]
	Hence 
	\[
	a(t)+b(t) \leq 2C_1 \epsilon
	\]
	for all $0\leq t \leq \exp{\left(\frac{1}{4C_1C_2 \epsilon}-1\right)}-1$ where we note that for $\epsilon> 0 $ sufficiently small, the last number satisfies the required properties in the statement of the Theorem.
 \end{proof}

\section{Related equations}
In this section we want to show that the results obtained for the harmonic map heat flow can be extended to two other parabolic equations to either simplify proofs of existing results or yield new results. We note that in the case of the Navier-Stokes equation similar results have been obtained earlier by Cannone \cite{can}, Planchon \cite{planch} and Koch-Tataru \cite{KT}. 

\subsection{Equations with power-type nonlinearities}
Here we study the very classical semilinear equation
\begin{align}
\partial_t u-\Delta u = f(u) \label{semilinear}
\end{align}
where $u:\R^d\times [0,\infty) \to \R$ with $u(\cdot,0)=u_0$ and $f:\R\to \R$ is a smooth function satisfying $|f(s)|\leq |s|^q$ for some $q>1$ and all $s\in \R$. It was already shown in  \cite{Fujita,Wei81} that for non-negative initial data $u_0$ so that the norm $\|u_0\|_{L^{\frac{d(q-1)}{2}}}$ is sufficiently small (here we also assume $\frac{d(q-1)}{2}>1$), there exists a non-negative global solution of \eqref{semilinear} satisfying
\[
\sup_{t>0}t^{\frac{1}{q-1}-\frac{d}{2p}}\|u(t)\|_{L^p} \leq C \|u_0\|_{L^{\frac{d(q-1)}{2}}}
\]
for every $\frac{d(q-1)}{2}< p < \frac{dq(q-1)}{2}$. This result should be directly compared (and was in fact a motivation) to the one of Soyeur \cite{Soyeur} (see Theorem \ref{soy}) for the harmonic map heat flow. In fact, with the same technique that we used in the previous section, we are able to extend these results and we obtain
\begin{theorem}\label{newweissler}
There exists a constant $\varepsilon>0$ depending only on $d$ and $q$ such that if $u_0\in B^{\frac{d}{p}-\frac{2}{q-1}}_{p,\infty} (\R^d)$ with $1<q<\infty$, $1< \frac{d(q-1)}{2}$ and if $\|u_0\|_{B^{\frac{d}{p}-\frac{2}{q-1}}_{p,\infty}} <\varepsilon$, then we have the estimate
\[
\sup_{t>0} t^{\frac1{q-1}-\frac{d}{2p}}\|u(t)\|_{L^p} \leq C \|u_0\|_{B^{\frac{d}{p}-\frac{2}{q-1}}_{p,\infty}},
\]
where $C=C(d,p)$, for every $\frac{d(q-1)}{2}< p < \frac{dq(q-1)}{2}$ and every solution $u:(0,\infty)\times \R^d\to \R$ of \eqref{semilinear} with
\[
t^{\frac1{q-1}-\frac{d}{2p}} u(t,x) \in C^0((0,\infty),L^p(\R^d))
\]
and
\[
\lim_{t\searrow 0} t^{\frac1{q-1}-\frac{d}{2p}}\|u(t)\|_{L^p} =0.
\]
If we know additionally that $q>\max \{2,1+\frac{2}{d}\}$ then we also get the estimate
\[
\sup_{t>0} t^{\frac1{q-1}}\|u(t)\|_{L^\infty} \leq C \|u_0\|_{B^{\frac{d}{p}-\frac{2}{q-1}}_{p,\infty}}.
\]
\end{theorem}
\begin{proof}
Arguing as in the proof of Theorem \ref{newsoy}, we obtain
\begin{align*}
    \|u(t)-G(t)u_0\|_{L^p}\leq& C\int_0^t \||u(s)|^q\|_{L^{\frac{p}{q}}} (t-s)^{-\frac{d(q-1)}{2p}}\, ds\\
    \leq&  C\int_0^t \|u(s)\|_{L^p}^q(t-s)^{-\frac{d(q-1)}{2p}}\, ds\\
    \leq& Cm(T)^q\int_0^t s^{\frac{dq}{2p}-\frac{q}{q-1}}  (t-s)^{-\frac{d(q-1)}{2p}}\, ds\\
    \leq& Cm(T)^q,
\end{align*}
since $0>\frac{dq}{2p}-\frac{q}{q-1} >-1 $ and $ 0 > -\frac{d(q-1)}{2p}>-1 $. Note that here we let
\[
m(T)=\sup_{0\leq t\leq T} t^{\frac1{q-1}-\frac{d}{2p}}\|u(t)\|_{L^p}.
\]
As above, this estimate implies the desired claim since
\[
\sup_{t>0} t^{\frac1{q-1}-\frac{d}{2p}} \|G(t)u_0\|_{L^p} \cong \|u_0\|_{B^{\frac{d}{p}-\frac{2}{q-1}}_{p,\infty}} .
\]

Next, if $q>2$ there exists a $p$ satisfying $\frac{dq}{2}<p<\frac{dq(q-1)}{2}$ and for one such value of $p$ we note that for every $0<t<\infty$
\begin{align*}
  \|u(t)-G(t)u_0\|_{L^\infty}\leq& C\int_0^t \||u(s)|^q\|_{L^{\frac{p}{q}}} (t-s)^{-\frac{dq}{2p}}\, ds\\  
  \leq&  C\int_0^t \|u(s)\|_{L^p}^q(t-s)^{-\frac{dq}{2p}}\, ds\\
    \leq& Cm(T)^q\int_0^t s^{\frac{dq}{2p}-\frac{q}{q-1}}  (t-s)^{-\frac{dq}{2p}}\, ds\\
    \leq& Ct^{-\frac{1}{q-1}}m(T)^q,
\end{align*}
and hence we also get the estimate
\[
\sup_{t>0} t^{\frac1{q-1}}\|u(t)\|_{L^\infty} \leq C \|u_0\|_{B^{\frac{d}{p}-\frac{2}{q-1}}_{p,\infty}}.
\]
\end{proof}
Note that the $L^p$-estimate in the previous theorem has been obtained earlier by Miao, Yuan and Zhang \cite{MYZ}. Related estimates have also been obtained by Blatt and Struwe \cite{BS} for small initial data in Morrey spaces.

\subsection{Biharmonic map flow}

As another example 	we consider solutions $ u : \mathbb{R}^d \times (0,T) \to S^{m-1}\subset \R^m $ of
the (extrinsic) biharmonic map heat flow governed by
\begin{align}\label{start2}
(\partial_t+\Delta^2) u &= u(|\Delta u|^2-\Delta |\nabla u|^2-2 \text{div} \langle \Delta u, \nabla u\rangle)\nonumber \\
&= u|\Delta u|^2+ \nabla u \cdot (\nabla |\nabla u|^2+2\langle \Delta u,\nabla u\rangle ) -\text{div} (u \nabla |\nabla u|^2+2u\langle \Delta u, \nabla u\rangle ) \\
&=: f_1[u]+ \text{div} f_2[u], \nonumber
\end{align}
where we can estimate
\begin{align*}
|f_1[u]| &\leq C(|\nabla^2 u|^2 + |\nabla u|^4) \ \ \ \text{and} \\
|f_2[u]| &\leq C|\nabla u| |\nabla^2 u|.
\end{align*}
For the sake of simplicity we restrict ourselves to sphere-valued maps. The arguments below can easily be extended to general closed manifolds $N$ and maps $u : \mathbb{R}^d \times (0,T) \to N$. Additionally, there exists another version of the biharmonic map flow, the so called intrinsic biharmonic map. All of the results shown below carry over directly to this fourth order parabolic PDE.

Again we are interested in an unconditional uniqueness resp. stability result for small initial data in a suitable Besov space.
As before, the principle of linear stability
cannot be applied directly due to the fact that the linearization again possesses an essential spectrum up to the imaginary axis. 

But as before some diffusive behavior can be used to 
control the nonlinear terms which are irrelevant w.r.t. this diffusive behavior. In the following we let $b:\R^d\times (0,\infty) \to \R$  denote the biharmonic heat kernel. It follows from the results of \cite{KochLamm} that Lemma \ref{kernel} can be extended to this case, namely we have
\begin{lemma}\label{bihkernel}
Let $b: \R^d\times (0,\infty)\to \R$ be the biharmonic heat kernel. Using the notation $b(t)f=b_t \star f$ we have the estimate
\begin{align}
\| \nabla^\sigma b (t)  \|_{\mathcal{L}(L^r,L^q)} \leq C t^{-\frac{d}{4}(\frac{1}{r}-\frac{1}{q})-\frac{\sigma}{4}} , \label{heatbih}
\end{align}
where $1\leq r\leq q\leq \infty$, $\sigma\in \N_0$ and $C<\infty$ is a constant.
\end{lemma}
Before we now state our main results for the biharmonic map flow, we note that we have another equivalent characterization of the homogeneous Besov space via the biharmonic heat kernel given by
\[
\|u_0\|_{\dot{B}^{\frac{d}{p}}_{p,\infty}}:=\sup_{t>0} \left( t^{\frac14 -\frac{d}{4p}} \|\nabla (b(t)u_0)\|_{L^p}+ t^{\frac12 -\frac{d}{2p}} \|\nabla^2 (b(t)u_0)\|_{L^{\frac{p}{2}}}\right).
\]
Here we note that the first term on the right hand-side indeed defines an equivalent norm on the homogeneous Besov space $\dot{B}^{\frac{d}{p}}_{p,\infty}$, whereas the second term defines an equivalent norm on the homogeneous space $\dot{B}^{\frac{2d}{p}}_{\frac{p}{2},\infty}$. Since the second homogeneous Besov space is embedded in the first one we choose our norm as a definition of the larger space.

\begin{theorem}\label{newsoybih}
There exists a constant $\varepsilon>0$ depending only on $d$ and $n$ such that if $u_0\in \dot{B}^{\frac{d}{p}}_{p,\infty}(\R^d,S^{m-1})$ with $d<p<\infty$ and $\|u_0\|_{\dot{B}^{\frac{d}{p}}_{p,\infty}} <\varepsilon$, then we have the estimate
\[
t^{\frac14-\frac{d}{4p}}\|\nabla u(t)\|_{L^p} +t^{\frac12-\frac{d}{2p}}\|\nabla^2 u(t)\|_{L^{\frac{p}{2}}} \leq C\| u_0\|_{\dot{B}^{\frac{d}{p}}_{p,\infty}},
\]
where $C=C(d,m,p)$, for every solution $u:(0,\infty)\times \R^d\to S^{m-1}$ of \eqref{start1} with 
\[
t^{\frac14-\frac{d}{4p}}\nabla u \in C^0([0,\infty), L^{p}(\R^d)), \, \, t^{\frac12-\frac{d}{2p}}\nabla^2 u \in C^0([0,\infty), L^{{\frac{p}{2}}}(\R^d))
\]
and 
\[
\lim_{t\searrow 0} (t^{\frac14-\frac{d}{4p}} \|\nabla u(t)\|_{L^p} +t^{\frac12-\frac{d}{2p}}\|\nabla^2 u(t)\|_{L^{\frac{p}{2}}})=0.
\]
\end{theorem}
\begin{proof}
We follow closely the proof of Theorem \ref{newsoy} by observing that
\begin{align*}
\nabla^i (u(x,t)-b(t) \star u_0 (x)) =& \int_0^t \int_{\R^d} \nabla^i b(x-y,t-s) f_1[u]\, dy\, ds\\
&- \int_0^t \int_{\R^d} \nabla^{i+1} b(x-y,t-s) f_2[u](s)\, dy\, ds=:I_i+II_i.
\end{align*}
Using the estimate \eqref{heatbih} with $\sigma=1$, $r=\frac{p}{4}$ and $q=p$ we get for every $ t>0$
\begin{align*}
\|I_1\|_{L^p} \leq& C\int_0^t \left( \| |\nabla u(s)|^4 \|_{L^{\frac{p}{4}}}+\| |\nabla^2 u(s)|^2\|_{L^{\frac{p}{4}}} \right) (t-s)^{-(\frac14 +\frac{3d}{4p})}\, ds\\
\leq& C\int_0^t (\| \nabla u(s) \|^4_{L^p}+\|\nabla^2 u(s)\|^2_{L^{\frac{p}{2}}}) (t-s)^{-(\frac14 +\frac{3d}{4p})} \, ds.
\end{align*}
Setting 
\[
m(T) =\sup_{0\leq t \leq T} \left(t^{\frac14 -\frac{d}{4p}} \|\nabla u(t)\|_{L^p}+ t^{\frac12 -\frac{d}{2p}} \|\nabla^2 u(t)\|_{L^{\frac{p}{2}}}\right),
\]
the last inequality implies for all $t>0$
\begin{align*}
 t^{\frac14 -\frac{d}{4p}} \|I\|_{L^p}\leq& Ct^{\frac14 -\frac{d}{4p}} m(T)^2(1+m(T)^2) \int_0^t (t-s)^{-(\frac14 +\frac{3d}{4p})} s^{-1+\frac{d}{p}} \, ds\\
 \leq& Cm(T)^2(1+m(T)^2),
 \end{align*}
 where we used again the identity \eqref{heat2}. 
 
 Similarly, using \eqref{heatbih} with $\sigma=2$, $r=\frac{p}{3}$ and $q=p$ we get for every $ t>0$
 \begin{align*}
 \|II_1\|_{L^p} \leq& C\int_0^t  \| |\nabla u(s)| |\nabla^2 u(s)| \|_{L^{\frac{p}{3}}}  (t-s)^{-(\frac12 +\frac{d}{2p})}\, ds\\
\leq& C\int_0^t (\| \nabla u(s) \|_{L^p}\|\nabla^2 u(s)\|_{L^{\frac{p}{2}}}) (t-s)^{-(\frac12 +\frac{d}{2p})} \, ds.
\end{align*}
Thus, we obtain
 \begin{align*}
 t^{\frac14 -\frac{d}{4p}} \|II\|_{L^p}\leq& Ct^{\frac14 -\frac{d}{4p}} m(T)^2 \int_0^t (t-s)^{-(\frac12 +\frac{d}{2p})} s^{-\frac34+\frac{3d}{4p}} \, ds\\
 \leq& Cm(T)^2,
 \end{align*}
 which implies
 \begin{align*}
 t^{\frac14 -\frac{d}{4p}}  \|\nabla u(t)\|_{L^p}\leq C  \| u_0\|_{\dot{B}^{\frac{d}{p}}_{p,\infty}}+Cm(T)^2(1+m(T)^2).
 \end{align*}
For $i=2$ we argue similarly, only this time we choose $\sigma=2$, $r=\frac{p}{4}$ and $q=\frac{p}{2}$ for $I_2$ resp. $\sigma=3$, $r=\frac{p}{3}$ and $q=\frac{p}{2}$ for $II_2$ to obtain
\begin{align*}
 t^{\frac12 -\frac{d}{2p}}  \|\nabla^2 u(t)\|_{L^{\frac{p}{2}}}\leq C  \| u_0\|_{\dot{B}^{\frac{d}{p}}_{p,\infty}}+Cm(T)^2(1+m(T)^2).
 \end{align*} 
 Taking the supremum over all $0\leq t \leq T$ we thus obtain for every $0<T<\infty$
 \[
 Cm(T)^2(1+m(T)^2) -m(T)+C \|u_0\|_{\dot{B}^{\frac{d}{p}}_{p,\infty}} \geq 0
 \]
 and for $\varepsilon>0$ small enough this finally implies for all $T<\infty$ that 
 \[
 m(T) \leq C\|u_0\|_{\dot{B}^{\frac{d}{p}}_{p,\infty}} \leq C\varepsilon.
 \]
 Here we used that $m(0)=0$. By definition this estimate implies that
 \[
 t^{\frac14 -\frac{d}{4p}}  \|\nabla u(t)\|_{L^p}+t^{\frac12- \frac{d}{2p}}\|\nabla^2 u(t)\|_{L^{\frac{p}{2}}}  \leq C \|u_0\|_{\dot{B}^{\frac{d}{p}}_{p,\infty}} \]
 for every $0<t<\infty$, as claimed.
\end{proof}

As in the case of the harmonic map heat flow this result can now be used to improve the uniqueness result for solutions of the biharmonic map heat flow with small initial data in $BMO$ due to Wang \cite{wangbih}. For this we have to define again a suitable function space. This time it is given by
\begin{align*}
\|u\|_{X_b}:=& \sup_{0<t<\infty}\left( \|u(t)\|_{L^\infty(\R^d)} +\sum_{i=1}^2t^{\frac{i}{4}} \|\nabla^i u(t)\|_{L^\infty(\R^d)}\right) \\
&+\sup_{x\in \R^d} \sup_{0<R<\infty}\sum_{i=1}^2 R^{-\frac{di}{4}}\|\nabla^i u\|_{L^{\frac{4}{i}}(B_R(x)\times(0,R^4))}
\end{align*}
is finite.

Similarly to Theorem \ref{main1} we then obtain
\begin{theorem}\label{main2}
There exists $\varepsilon>0$, $C>0$ such that for every $u_0\in \dot{B}^{\frac{d}{p}}_{p,\infty}(\R^d,S^{m-1})$ with $d<p<\infty$, $\|u_0\|_{\dot{B}^{\frac{d}{p}}_{p,\infty}} <\varepsilon$ and every solution $u:(0,\infty)\times \R^d\to S^{m-1}$ of \eqref{start2} with 
\[
t^{\frac14-\frac{d}{4p}}\nabla u(t,x) \in C^0([0,\infty), L^{p}(\R^d)), \, t^{\frac12-\frac{d}{2p}}\nabla^2 u(t,x) \in C^0([0,\infty), L^{\frac{p}{2}}(\R^d))
\]
and 
\[
\lim_{t\searrow 0} \left( t^{\frac14-\frac{d}{4p}} \|\nabla u(t)\|_{L^p} +t^{\frac12-\frac{d}{2p}} \|\nabla^2 u(t)\|_{L^{\frac{p}{2}}}\right)=0,
\]
we have the estimate
\[
\| u -b(t)u_0\|_{X_b}\leq C\varepsilon,
\]
i.e., there exists only one solution of \eqref{start2} under these assumptions. 

Additionally, one obtains the bounds
\begin{align*}
 \sup_{t>0} \sum_{i=1}^2t^{\frac{i}{4}} \|\nabla^i u(t)\|_{L^\infty(\R^d)} &+\sup_{x\in \R^d} \sup_{0<R<\infty}\sum_{i=1}^2 R^{-\frac{di}{4}}\|\nabla^i u\|_{L^{\frac{4}{i}}(B_R(x)\times(0,R^4))} \\ &\leq C\|u_0\|_{\dot{B}^{\frac{d}{p}}_{p,\infty}} .
\end{align*}
\end{theorem}

\section{Self-similar decay}\label{sss_rgd}

In this section we collect a number  of results which show an asymptotic self-similar decay 
of the deviations $ v $  for the harmonic and bi-harmonic map heat flow. We start with  the  
$ v $-equation \eqref{veq} for the  harmonic  map heat flow.
The transfer of the presented method to the biharmonic map heat flow can be found below.
We close this section with some remarks about similar statements for the 
$ w $-equation  \eqref{weq}.

\subsection{Harmonic map heat flow}

It is well known that for spatially localized initial conditions the solutions of the  linear diffusion  equation
$$
\partial_t v = \Delta v, \qquad v|_{t=0} = v_0
$$
decay asymptotically in a self similar way, i.e., for $ x \in \R^d $ the renormalized solution 
$
t^{d/2} v (x \sqrt{t},t) $ converges towards a multiple of a Gaussian,  
$  V_{\rm lim} e^{-|x|^2/4}  $ with $ V_{\rm lim} \in \R $.  The same is true for the deviation $ v $ of a trivial spatially constant equilibrium $ u^* $ for the harmonic map heat flow \eqref{start1}. As we have seen in the introduction, the deviation $ v $ 
satisfies a semilinear diffusion equation of the form 
\begin{equation}  \label{start1zz}
\partial_t v-\Delta v =A(u^*+v)(\nabla v,\nabla v), 
\end{equation}
where we recall that $ A(u^*+v)(\cdot,\cdot) $ is a smooth bounded bilinear mapping acting on the $ d$-dimensional tangent space.
Our first result reads as follows
\medskip

\begin{theorem}\label{th41}
There exist 
$ \delta, C  > 0 $, such  that  for all  solutions $v $  of  \eqref{start1zz}
with  $ \| v|_{t = 0 } \|_{H^m_r}  \leq  \delta$ where $ m > d/2 +1 $ and $ r > d/2 + 1$
we have a  $ V_{\rm lim} \in \R^d $ such that
$$
\| t^{d/2}\, v( \cdot \sqrt{t},t) -
V_{\rm lim} e^{- |\cdot|^2/4} \|_{H^m_r}
\leq C (1{+}t)^{-d/2} 
$$
for all  $  t \geq 0 $,
where 
$$
\| v \|_{H^m_r} = \| v \sigma^{r} \|_{H^m} , \quad \textrm{with} \quad \sigma(x) = (1+x^2)^{1/2}.
$$
\end{theorem}
\noindent
\begin{proof}
We use   the discrete renormalization approach which was introduced for such problems  at the beginning of the 1990s, cf. \cite{BK92,BKL94}.
Although the proof is documented in the literature for similar problems,
we  recall its main steps 
since to our knowledge such results have not been 
formulated for geometric flow problems so far.

	Instead of solving  \eqref{start1zz} directly we consider an equivalent  sequence of 
	problems which converges formally towards the linear diffusion equation.
	For this we let
\begin{equation} \label{rgscaling}
v_n(\xi,\tau) = L^{n d}v(L^n \xi ,L^{2n}\tau) ,
\end{equation}
with  $L > 1$ fixed and  $n \in \N$. Then  $v_n$ satisfies 
\begin{eqnarray*}
\lefteqn{(\partial_{\tau} v_n- \Delta_{\xi} v_n)=  L^{n(d + 2) } (\partial_t v-\Delta_x v)} \\
&=& L^{n(d+2)} A(u^*+v)(\nabla_x v , \nabla_x v) \\
&=& L^{n(d+2)} A(u^*+L^{-n d } v_n) 
(L^{-n(d+1)}  \nabla_{\xi} v_n, L^{-n(d+1)}  \nabla_{\xi} v_n)  \\
& = & 
 L^{-nd}A(u^*+L^{-n d} v_n) 
( \nabla_{\xi} v_n,  \nabla_{\xi} v_n) 
\end{eqnarray*}
for $ \tau \in [L^{-2},1] $,
with initial condition
\begin{equation}
v_n(\xi,L^{-2})=L^d v_{n-1}(L \xi,1).
\end{equation}
The influence of the nonlinear terms vanishes
with a geometric rate for  $n \to \infty$. In the limit $n \to \infty$ we obtain 
a linear diffusion equation.
With a simple  fixed point argument  
the convergence of  
$v_n|_{\tau = 1} $  towards the limit of the renormalized
linear diffusion  problem, namely multiples of the  Gaussian $  \psi(\xi) = (2 \pi)^{-d/2}e^{-|\xi|^2/4}$,
can be established.

{\bf Notation.}
In the rest of the proof we use  the symbol  $ C $ for constants which can be chosen  independent  of $ L $ and $ n $.

 We need the  projection 
$$ 
\Pi v_n = \int v_n(\xi) d\xi ,
$$ 
respectively   $  \widehat{\Pi} \widehat{v}_n = \widehat{v}_n|_{k=0} $ in Fourier space. 
The bound 
$$
|\Pi v_n |  \leq C | \widehat{v}_n|_{k=0} | \leq  C \|  \widehat{v}_n \|_{C^0_b}
 \leq  C \|  \widehat{v}_n \|_{H^{r}_m}
 \leq  C \| v_n \|_{H^{m}_r}
$$
holds due to Sobolev's embedding theorem $ H^r_m \subset C^0_b $  for $ r > d/2 $.
We set 
$$ 
V_{n+1} =  \Pi {v}_n|_{\tau = 1}
$$
and introduce the deviation of $ {v}_n(\xi,1)$  from $V_{n+1} \psi(\xi) $ by
$$
{\rho}_{n+1}(\xi) = {v}_n(\xi,1)  - V_{n+1} \psi(\xi) .
$$
By construction we have 
$ \int \rho_n(\xi) d\xi = 0 $, resp.
$ \widehat{\rho}_n|_{k=0} = 0 $. For proving the convergence of $ (V_n)_{n \in \N} $ towards a $ V_{\rm lim} \in \R$ and the decay of 
$ (\rho_n)_{n \in \N}  $ towards zero
we establish a number of inequalities.
\begin{lemma} \label{lem42}
We have 
\begin{eqnarray*}
|V_{n+1}-V_n | & \leq &  C L^{-n d } R_n^{2}  ,\\
\|{\rho}_{n+1}\|_{H^m_r}  & \leq & (C/L) \| {\rho}_n \|_{H^m_r} +
C L^{-n d } R_n^{2} + C |V_{n+1}-V_n | 
\end{eqnarray*}
where 
$$ R_n = \sup_{\tau \in [1/L^2,1]} \| {v}_n(\tau) \|_{H^m_r}.$$ 
\end{lemma}
\begin{proof}
We consider 
the  variation of  constants formula 
\begin{equation} \label{vdk17}
{v}_{n}(\xi,\tau)  =  e^{(\tau-1/L^2)\Delta_{\xi}} L^d {v}_{n-1}(L \xi,1) + g_n(\tau) 
\end{equation}
for $ \tau \in [1/L^2,1] $, where 
$$ 
g_n(\tau)= 
L^{-nd} \int_{1/L^2}^{\tau} e^{(\tau-s)\Delta_{\xi}}  A(u^*+L^{-nd } v_n) 
( \nabla_{\xi} v_n,  \nabla_{\xi} v_n)  
 (s) d s .
 $$  
Applying the  projection $ \Pi $ in case $ \tau = 1 $ yields
\begin{equation}\label{Veq}
V_{n+1}  =  V_n +
\Pi g_n(1),
\end{equation}
where we used 
$$
 \int  e^{(\tau-1/L^2)\Delta_{\xi}} f(\xi)  d\xi =  e^{-(\tau-1/L^2)k^2}|_{k=0} \widehat{f}(0) 
 = \widehat{f}(0) =  \int f(\xi)  d\xi
$$ 
and 
$$
 \int   L^d {v}_{n-1}(L \xi,1) d\xi 
 = \int {v}_{n-1}(\xi,1) d\xi = V_n.
$$
Rearranging the terms in the variation of constant formula additionally yields 
\begin{eqnarray} \label{rhoeq}
{\rho}_{n+1}(\xi) &  = &  e^{(1-1/L^2)\Delta_{\xi}} L^d {\rho}_{n}(L \xi) \\ && +
g_n(1) +
e^{(1-1/L^2)\Delta_{\xi}} L^d V_n \psi(L \xi) - V_{n+1} \psi(\xi), \nonumber
\end{eqnarray}
We estimate the terms on the right hand side
of \eqref{Veq} and \eqref{rhoeq}.

{\bf i)} Using  
$$
e^{- (1-1/L^2)|k|^2}e^{-|k/L|^2} = e^{-|k|^2}
$$
yields
\begin{align*}
\|  e^{(1-1/L^2)\Delta_{\xi}} L^d V_n \psi(L \cdot) - V_{n+1} \psi(\cdot) \|_{H^m_r} 
 =&  \| (V_n- V_{n+1}) \psi(\cdot) \|_{H^m_r} \\
 \leq& C |V_{n+1} - V_n | .
\end{align*}

{\bf ii)} Since $ \rho $ has mean value zero and since $ \widehat \rho \in H^r_m \subset C^1_b  $ 
for $ r > d/2 +1 $
we have
\begin{eqnarray*}
\|  e^{(1-1/L^2) \Delta_{\xi}} L^d  {\rho}_{n}(L \cdot) \|_{H^m_r} 
&\leq & 
C \|  e^{- (1-1/L^2) |k|^2} \widehat{\rho}_{n}(\cdot/L) \|_{H^{r}_m}  \\ 
& \leq &
(C/L) 
\|
\widehat{\rho}_{n} \|_{H^{r}_m}
\leq 
(C/L) \|
{\rho}_{n} \|_{H^m_r}   .
\end{eqnarray*}

{\bf iii)}
For the integral term we obtain  
\begin{eqnarray*}
\| g_n(\tau)  \|_{H^m_r} & \leq & L^{-nd} \int_{1/L^2}^{\tau} \| e^{(\tau-s)\Delta}  
 \|_{H^{m-1}_2 \to H^m_r} \\ &&\qquad \qquad
  \times \| A(u^*+L^{-nd } v_n) 
( \nabla_{\xi} v_n,  \nabla_{\xi} v_n)   (s) \|_{H^{m-1}_2}
 d s 
\\  
  & \leq & C L^{-nd} R_n^2 \int_{1/L^2}^{\tau}   (\tau-s)^{-1/2} ds   \leq C  L^{-nd} R_n^2.
\end{eqnarray*}
\end{proof}

Using the first inequality of Lemma \ref{lem42} to eliminate $ |A_{n+1}-A_n |$ in the second inequality of Lemma \ref{lem42} gives
\begin{corollary} \label{coro43}
\begin{eqnarray*} 
|V_{n+1}-V_n | & \leq &  C L^{-n d  } R_n^{2}  ,\\
\| {\rho}_{n+1}\|_{H^m_r}  & \leq & (C/L) \| {\rho}_n \|_{H^m_r} +
C L^{-n d } R_n^{2}  .
\end{eqnarray*}
\end{corollary}
For obtaining a closed system of inequalities 
in $ V_n $ and $ \rho_n $, we have to estimate 
$ R_n $ in terms of  $ V_n $ and $ \rho_n $. 
\begin{lemma}         \label{lem7n}
There exists $ C_1> 0 $ and $ C_2> 0  $, such that for all $ \delta > 0 $ and $ L > 1 $ 
with  $ L^{5/2} \delta < C_1 $ the following holds.
For 
 $ {v}_{n-1}|_{\tau=1} \in H^m_r $ with
$$ \| {v}_{n-1}|_{\tau=1} \|_{H^m_r} < \delta $$ 
we have 
$$
R_n
\leq C_2 L^{5/2} ( |V_n| + r_n )
$$
where $ r_n = \| {\rho}_{n}\|_{H^m_r} $.
\end{lemma}
\noindent
\begin{proof}
We consider the  variation of  constant formula \eqref{vdk17}, 
use the estimate iii) from Lemma \ref{lem42}, replace ii) from Lemma \ref{lem42} by 
\begin{eqnarray*}
\|  e^{(\tau-1/L^2) \Delta_{\xi}} L^d  {v}_{n-1}(L \cdot) \|_{H^m_r} 
\leq 
C \|  L^d  {v}_{n-1}(L \cdot) \|_{H^m_r} 
 \leq C L^{5/2} \| v_{n-1} \|_{H^m_r} ,
\end{eqnarray*}
to obtain
$$ 
R_n \leq C L^{5/2}  \| v_{n-1} \|_{H^m_r}  + C L^{-\sigma n} R_n^2.
$$
Since there exist two intersection points $ \underline{R} < \overline{R} $
of
the line $ R_n \mapsto R_n $ and the curve 
$ R_n \mapsto C L^{5/2} \| v_{n-1} \|_{H^m_r} + C L^{-n d} R_n^2 $ for $ L^{5/2} \delta > 0 $ sufficiently 
small, we have 
$$ 
R_n \leq \underline{R} \leq 2 C L^{5/2} \| v_{n-1} \|_{H^m_r} \leq 4 C L^{5/2}  ( |V_n| + r_n ).
$$
Obviously there is local existence and uniqueness of solutions in the space  $ H^m_r $. 
Since the above estimate gives an a priori bound in  $ H^m_r $, the solution 
can be extended  to the whole interval $ [1/L^2,1] $. 
\end{proof}

Corollary \ref{coro43} and Lemma \ref{lem7n} yield
\begin{eqnarray*}
|V_{n+1}-V_n | & \leq &  C L^{-n d} (L^{5/2})^2 (|V_n| + r_n)^2   ,\\
r_n & \leq & (C/L) r_n  +
C L^{-n d} (L^{5/2})^2 (|V_n| + r_n)^2 , 
\end{eqnarray*}
as long as 
\begin{equation} \label{cond14eq}
 C ( |V_n| + r_n ) \leq \delta   \qquad \textrm{and} \qquad  L^{5/2} \delta \leq C_1 . 
 \end{equation}
We choose 
 $ L_0 > 1 $ so large that $ (C/L_0) \leq 1/10 $ and then $ n_0 > 0 $ so large that 
 $ C L_0^{-n_0 d} (L_0^{5/2})^2  \leq 1$.
 Then, for all 
$ n > n_0 $ und $ L > L_0 $ we have 
\begin{eqnarray*}
|V_{n+1}-V_{n}| & \leq & 
L^{-d(n-n_0)}  (V_n + r_n)^2 ,\\
r_{n+1} & \leq &   r_{n}/10+ L^{-(n-n_0)d}  (V_n + r_n)^2.
\end{eqnarray*}
Now we choose an  $ L > L_0 $ and a $ \delta > 0 $ so small, that 
$$ 
L^{5/2} \delta \leq C_1.
$$  
Since the quantities 
$ |V_n | $ and  $ r_n $
grow only for  $ n_0 $ steps
we can choose the initial conditions so small that
$$ 
\max_{n=1,\ldots,n_0} C ( |V_n| + r_n ) \leq \delta/4 .
$$
Then
 $(V_n)_{n \in \N}$ converges with a geometric rate towards a
$V_{\rm lim} \in \R^d $ and we have
$ \lim_{n \rightarrow \infty} r_n = 0 $.
The estimates further guarantee that  
$$ 
\sup_{n \geq n_0} r_n \leq \delta/4  \qquad \textrm{und} \qquad \sup_{n \geq n_0} |V_n| \leq \delta/2.
$$ 
Therefore, all necessary conditions for the validity of the previous estimates are satisfied. 
Since  convergence holds for all  
 $ L \in [L_0,L_0^2]  $, Theorem \ref{th41} follows. 
\end{proof}

\bigskip

With a slight modification of this proof similar results can  be   established for  the $ w $-equation of the harmonic map heat flow (see \eqref{weq}) and  for the 
biharmonic map heat flow (see \eqref{start2}).  In the following  we sketch the ideas for the equation \eqref{weq} which we recall to be
\begin{eqnarray}  \label{start1zzz}
\partial_t w-\Delta w &  = & \nabla (A(u^*+v)( w,w)) \\ & = & 2 A(u^*+v)( w, \nabla w) + A'(u^*+v) (w ,w,w) ,  \nonumber
\end{eqnarray}
with $ A'(u^*+v) (\cdot ,\cdot,\cdot) $ a smooth and bounded trilinear mapping acting on the tangent space and 
with $ w = \nabla v $. For 
the scaling
\begin{equation} \label{rgscalingw}
w_n(\xi,\tau) = L^{n d}w(L^n \xi ,L^{2n}\tau) 
\end{equation}
with  $L > 1$ fixed and  $n \in \N$ the previous  analysis applies almost line by line. 
By this choice we have 
$$ 
v_n(\xi,\tau) = L^{n (d-1)}v(L^n \xi ,L^{2n}\tau) ,
$$
and so   $w_n$ satisfies 
\begin{eqnarray*}
\lefteqn{(\partial_{\tau} w_n- \Delta_{\xi} w_n)=  L^{n(d + 2) } (\partial_t w-\Delta_x w)} \\
&=& 2 L^{n(d+2)}  A(u^*+v)( w, \nabla_x w) + L^{n(d+2)} A'(u^*+v) (w ,w,w)  \\
&=& 2 L^{n(d+2)}  A(u^*+L^{-n (d-1)} v_n)( L^{-n d} w_n, L^{-n (d+1)}\nabla_{\xi} w_n) \\ && + L^{n(d+2)} A'(u^*+L^{-n (d-1)} v_n) (L^{-n d} w_n ,L^{-n d} w_n,L^{-n d} w_n)  \\
&=& 2 L^{-n(d-1)}  A(u^*+L^{-n (d-1)} v_n)(w_n, \nabla_{\xi} w_n) \\ 
&& + L^{- 2 n(d-1)} A'(u^*+L^{-n (d-1)} v_n) (w_n , w_n,w_n)  \end{eqnarray*}
for $ \tau \in [L^{-2},1] $,
with
\begin{equation}
w_n(\xi,L^{-2})=L^d w_{n-1}(L \xi,1).
\end{equation}
Hence, the nonlinear terms are asymptotically irrelevant if $ d > 1 $ and so the rest of the previous 
proof allows us to establish the result
\begin{theorem}\label{th42}
There exist 
$ \delta, C  > 0 $, such  that  for all  solutions $w $  of  \eqref{start1zzz}
with  $ \| w|_{t = 0 } \|_{H^m_r}  \leq  \delta$ where $ m > d/2 +1 $ and $ r > d/2 +1 $,
we have a  $ W_{\rm lim} \in \R^{d \times d} $ such that
$$
\| t^{d/2}\, w( \cdot \sqrt{t},t) -
W_{\rm lim} e^{- |\cdot|^2/4} \|_{H^m_r}
\leq C (1{+}t)^{-d/2} \textrm{  for all  } t \geq 0.
$$
\end{theorem}

However, since $ w = \nabla v $  the choice $ w \in H^s_r $ excludes a number 
of interesting solutions.
A typical example of an initial condition for $ v:\R^2 \to \R^2 $ considered in  Theorem \ref{th42} which 
can not have been handled by Theorem \ref{th41} or Theorem \ref{th42} is 
$$ 
v_1(x_1,x_2) = v_0 \textrm{erf}(x_1) e^{-x_2^2}.
$$  
We find
$$ 
(w_{1j})_{j=1,2}(x_1,x_2) = \nabla v(x_1,x_2)  = (v_0 e^{-x_1^2/4} e^{-x_2^2},- 2 x_2 v_0 \textrm{erf}(x_1) e^{-x_2^2})
$$ 
which is not an element of $ H^s_r $. We remark without any proof that initial conditions of the form 
$ v = c_1 \textrm{erf}(x_1) + \widetilde{v}$ with $ c_1 \in \R $ and $ \widetilde{v} \in H^s_r$
can be included easily into the formulation and proof of Theorem \ref{th41}.

\subsection{Biharmonic map heat flow}

Next we come to the biharmonic map heat flow in case $ u : \mathbb{R}^d \to S^{m-1} \subset \R^m $.
Similar to the linear diffusion equation we have for the linearized problem 
$$
\partial_t v = -\Delta^2 v, \qquad v|_{t=0} = v_0
$$
that for spatially localized initial conditions the solutions 
decay asymptotically in a self similar way, i.e., for $ x \in \R^d $ the renormalized solution 
$
t^{d/4} v (x \sqrt{t},t) $ converges towards a multiple of a universal  limit function,  
$  V_{\rm lim} \varphi(x) $ with $ V_{\rm lim} \in \R $ and $ \varphi = \mathcal{F}^{-1} \widehat{\varphi} $
where $ \widehat{\varphi}(k) = e^{-|k|^4} $.  The same is true for the deviation $ v $ of a trivial spatially constant equilibrium $ u^* $ for the bi-harmonic map heat flow \eqref{start1}. Note that the deviation $ v $ 
satisfies a semilinear parabolic equation of the form 
\begin{eqnarray} \label{start6zz}
	\partial_t v &=& - \Delta^2 v 
	-(\Delta (\nabla v,\nabla v)  +\nabla \cdot (\Delta v,\nabla v) 
	+ (\nabla \Delta v, \nabla v))e_1\\&&
	-(\Delta (\nabla v,\nabla v)  +\nabla \cdot (\Delta v,\nabla v) 
	+ (\nabla \Delta v, \nabla v))v .\nonumber
\end{eqnarray}
Then we have:
\begin{theorem}\label{th45}
There exist 
$ \delta, C  > 0 $, such  that  for all  solutions $v $  of  \eqref{start6zz}
with  $ \| v|_{t = 0 } \|_{H^m_r}  \leq  \delta$ where $ m > d/2 +3 $ and $ r > d/2 + 1$
we have a  $ V_{\rm lim} \in \R^d $ such that
$$
\| t^{d/4}\, v( \cdot \sqrt{t},t) -
V_{\rm lim} \varphi(\cdot) \|_{H^m_r}
\leq C (1{+}t)^{-d/4} 
$$
for all  $  t \geq 0 $.
\end{theorem}
\noindent
{\bf Proof.}
As above
	instead of solving  \eqref{start1zz} directly we consider an equivalent  sequence of 
	problems which converges formally towards the linearized  equation.
	For this we let
\begin{equation} \label{rgscalingbih}
v_n(\xi,\tau) = L^{n d}v(L^{2n} \xi ,L^{4n}\tau) ,
\end{equation}
with  $L > 1$ fixed and  $n \in \N$. Then  $v_n$ satisfies 
\begin{eqnarray*}
\lefteqn{(\partial_{\tau} v_n- \Delta_{\xi} v_n)=  L^{n(d + 8) } (\partial_t v-\Delta_x v)} \\
&=& L^{n(d+8)} (-(\Delta_x (\nabla_x v,\nabla_x v)  +\nabla_x \cdot (\Delta_x v,\nabla_x v) 
	+ (\nabla_x \Delta_x v, \nabla_x v))e_1\\&&
	-(\Delta_x (\nabla_x v,\nabla_x v)  +\nabla_x \cdot (\Delta_x v,\nabla_x v) 
	+ (\nabla_x \Delta_x v, \nabla_x v))v) \\
&=& L^{n(d+8)} (-(L^{-4n}\Delta_{\xi} (L^{-n(d+2)}\nabla_{\xi} v,L^{-n(d+2)} \nabla_{\xi} v) \\ &&  +L^{-2n}\nabla_{\xi} \cdot (L^{-n(d+4)}\Delta_{\xi} v,L^{-n(d+2)} \nabla_{\xi} v) \\&&
	+ (L^{-n(d+6)} \nabla_{\xi} \Delta_{\xi} v, L^{-n(d+2)}\nabla_{\xi} v))e_1\\&&
	-(L^{-4n} \Delta_{\xi} (L^{-n(d+2)}\nabla_{\xi} v,L^{-n(d+2)} \nabla_{\xi} v)  \\ &&
	+L^{-2n}\nabla_{\xi} \cdot (L^{-n(d+4)} \Delta_{\xi} v,L^{-n(d+2)} \nabla_{\xi} v) \\&&
	+ (L^{-n(d+6)}\nabla_{\xi} \Delta_{\xi} v,L^{-n(d+2)} \nabla_{\xi} v))L^{-nd} v)
\\	&=& L^{-nd} (-(\Delta_{\xi} (\nabla_{\xi} v,\nabla_{\xi} v) 
 +\nabla_{\xi} \cdot (\Delta_{\xi} v,\nabla_{\xi} v) 
	+ ( \nabla_{\xi} \Delta_{\xi} v, \nabla_{\xi} v)))e_1 \\ &&
	+L^{-2nd} (-( \Delta_{\xi} (\nabla_{\xi} v, \nabla_{\xi} v)  
	+\nabla_{\xi} \cdot ( \Delta_{\xi} v,\nabla_{\xi} v) 
	+ (\nabla_{\xi} \Delta_{\xi} v, \nabla_{\xi} v))v)
\end{eqnarray*}
for $ \tau \in [L^{-4},1] $,
with initial condition
\begin{equation}
v_n(\xi,L^{-4})=L^d v_{n-1}(L^2 \xi,1).
\end{equation}
As for the harmonic map heat flow the influence of the nonlinear terms vanishes
with a geometric rate for  $n \to \infty$. In the limit $n \to \infty$ we obtain 
the  linearized problem.
The rest of the proof follows the associated parts of the proof of Theorem \ref{th41}
almost line for line. \qed

\end{document}